\documentclass[12pt]{amsart}

\usepackage{epsfig} 
\usepackage{color}
\definecolor{darkgreen}{rgb}{0, 0.40, 0}

\usepackage{amsthm} 
\usepackage{amsmath} 
\usepackage{amssymb} 

\usepackage{psfrag} 

\usepackage[
ps2pdf,
]{hyperref}

\newcommand{\calA}{\mathcal{A}}

\newcommand{\calC}{\mathcal{C}}
\newcommand{\calD}{\mathcal{D}}

\newcommand{\calF}{\mathcal{F}}
\newcommand{\calG}{\mathcal{G}}
\newcommand{\calH}{\mathcal{H}}

\newcommand{\calK}{\mathcal{K}}
\newcommand{\calL}{\mathcal{L}}

\newcommand{\calN}{\mathcal{N}}

\newcommand{\calP}{\mathcal{P}}

\newcommand{\DD}{\mathbb{D}}
\newcommand{\EE}{\mathbb{E}}
\newcommand{\FF}{\mathbb{F}}
\newcommand{\GG}{\mathbb{G}}
\newcommand{\HH}{\mathbb{H}}
\newcommand{\II}{\mathbb{I}}

\newcommand{\RR}{\mathbb{R}}

\newcommand{\ZZ}{\mathbb{Z}}

\renewcommand{\setminus}{{\smallsetminus}}

\newcommand{\Id}{\operatorname{Id}}

\newcommand{\st}{\mathbin{\mid}} 
\newcommand{\from}{\colon} 

\newcommand{\homeo}{\mathrel{\cong}} 

\newcommand{\isom}{\cong} 
\newcommand{\cross}{\times}

\newcommand{\closure}[1]{{\overline{#1}}}

\newcommand{\neigh}{\operatorname{neigh}} 
\newcommand{\Neigh}{\closure{\neigh}} 
 
\newcommand{\bdy}{\partial} 
\newcommand{\dual}{\operatorname{dual}} 
\newcommand{\link}{\operatorname{link}}

\newcommand{\MCG}{\mathcal{MCG}} 
\newcommand{\Homeo}{\operatorname{Homeo}} 

\newcommand{\Aut}{\operatorname{Aut}} 

\newcommand{\PGL}{\operatorname{PGL}} 

\newcommand{\refsec}[1]{Section~\ref{Sec:#1}}
\newcommand{\refthm}[1]{Theorem~\ref{Thm:#1}}

\newcommand{\reflem}[1]{Lemma~\ref{Lem:#1}}

\newcommand{\refclm}[1]{Claim~\ref{Clm:#1}}

\newcommand{\reffig}[1]{Figure~\ref{Fig:#1}}
\newcommand{\refdef}[1]{Definition~\ref{Def:#1}}
\newcommand{\refprob}[1]{Problem~\ref{Prob:#1}}

\theoremstyle{plain}
\numberwithin{equation}{section} 
\newtheorem{theorem}[equation]{Theorem}

\newtheorem{lemma}[equation]{Lemma}

\theoremstyle{definition}
\newtheorem{definition}[equation]{Definition}

\newtheorem*{remark*}{Remark}
\newtheorem{claim}[equation]{Claim}
\newtheorem*{claim*}{Claim}
\newtheorem{problem}[equation]{Problem}

\newtheorem*{question*}{Question}
\newtheorem*{answer*}{Answer}
\newtheorem*{application*}{Application}

\newcommand{\fakeenv}{} 

\newenvironment{restate}[2]  
{ 
 \renewcommand{\fakeenv}{#2} 
 \theoremstyle{plain} 
 \newtheorem*{\fakeenv}{#1~\ref{#2}} 
 \begin{\fakeenv}
}
{
 \end{\fakeenv}
}

\newcommand{\join}{\mathbin{\vee}}
\newcommand{\cut}{\operatorname{cut}}
\newcommand{\NonSep}{\operatorname{NonSep}}
\newcommand{\CG}{\mathcal{CG}}
\newcommand{\handle}{\operatorname{handle}}
\newcommand{\pants}{\operatorname{pants}}

\begin{document}

\title{Automorphisms of the disk complex}

\author{Mustafa Korkmaz}
\address{\hskip-\parindent
        Department of Mathematics\\
	Middle East Technical University\\
	06531 Ankara, Turkey}
\email{korkmaz@arf.math.metu.edu.tr}

\author{Saul Schleimer}
\address{\hskip-\parindent
        Department of Mathematics\\
        University of Warwick\\
        Coventry, CV4 7AL, UK}
\email{s.schleimer@warwick.ac.uk}

\thanks{This work is in the public domain}

\date{\today}

\begin{abstract}
We show that the automorphism group of the disk complex is isomorphic
to the handlebody group.  Using this, we prove that the outer
automorphism group of the handlebody group is trivial.
\end{abstract}


\maketitle

\section{Introduction}

We show that the automorphism group of the disk complex is isomorphic
to the handlebody group.  Using this, we prove that the outer
automorphism group of the handlebody group is trivial.  These results
and many of the details of the proof are inspired by Ivanov's
work~\cite{Ivanov97} on the mapping class group and the curve complex.


Let $V = V_{g,n}$ be the genus $g$ handlebody with $n$ {\em spots}: a
regular neighborhood of a finite, polygonal, connected graph in
$\RR^3$ with $n$ disjoint disks chosen on the boundary.  See
\reffig{V_gnExample} for a picture of $V_{2,2}$.  We write $V = V_g$
when $n = 0$.  Let $\bdy_0 V$ denote the union of the spots.  Let
$\bdy_+ V$ be the closure of $\bdy V \setminus \bdy_0 V$.  So $\bdy_+
V \homeo S = S_{g,n}$ is a compact connected orientable surface of
genus $g$ with $n$ boundary components.  We write $S = S_g$ when $n =
0$.  Define $e(V) = -\chi(\bdy_+ V) = 2g - 2 + n$.

\begin{figure}[ht]
\psfrag{}{}
\psfrag{}{}
$$\begin{array}{c}
\epsfig{file=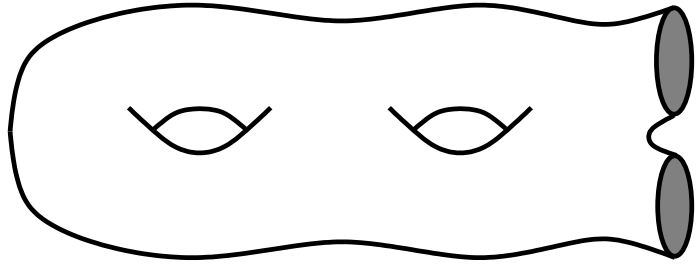, height = 3.5 cm}
\end{array}$$
\caption{A genus two handlebody with two spots.}
\label{Fig:V_gnExample}
\end{figure}

A simple closed curve $\alpha$ in $S = S_{g,n}$ is {\em inessential}
if it cuts a disk off of $S$; otherwise $\alpha$ is {\em essential}.
The curve $\alpha$ is {\em peripheral} if it cuts an annulus off of
$S$; otherwise $\alpha$ is {\em non-peripheral}.  A properly embedded
disk $D$ in $V = V_{g,n}$, with $\bdy D \subset \bdy_+ V$, is {\em
essential} or {\em non-peripheral} exactly as its boundary is in
$\bdy_+ V$.  We require any proper isotopy of $D \subset V$ to have
track disjoint from the spots of $V$.  This yields a proper isotopy of
$\bdy D$ in $\bdy_+ V$.

\begin{definition}[Harvey~\cite{Harvey81}]
\label{Def:CurveComplex}
The {\em curve complex} $\calC(S)$ is the simplicial complex with
vertex set being isotopy classes of essential, non-peripheral curves
in $S$.  The $k$--simplices are given by collections of $k+1$ vertices
having pairwise disjoint representatives.
\end{definition}

\begin{definition}[McCullough~\cite{McCullough91}]
\label{Def:DiskComplex}
The {\em disk complex} $\calD(V)$ is the simplicial complex with
vertex set being proper isotopy classes of essential, non-peripherial
disks in $V$.  The $k$--simplices are given by collections of $k+1$
vertices having pairwise disjoint representatives.
\end{definition}

Note that there is a natural inclusion $\calD(V) \to \calC(\bdy_+ V)$
taking a disk to its boundary.  This map is simplicial and injective.


If $\calK$ is a simplicial complex then $\Aut(\calK)$ denotes the
group of simplicial automorphisms of $\calK$.  The elements of
$\Aut(\calC(S))$ and $\Aut(\calD(V))$ are to be contrasted with
mapping classes on the underlying spaces.

\begin{definition}
\label{Def:MCG}
The {\em mapping class group} $\MCG(S)$ is the group of homeomorphisms
of $S$, up to isotopy.  The {\em handlebody group} $\calH(V)$ is the
group of homeomorphisms of $V$, fixing the spots setwise, up to spot
preserving isotopy.
\end{definition}

Some authors refer to our $\MCG(S)$ as the {\em extended} mapping
class group, as orientation reversing homeomorphisms are allowed.
Note there is a natural map $\calH(V) \to \MCG(\bdy_+ V)$ which takes
$f \in \calH(V)$ to $f|\bdy_+ V$.  Again, this map is an injective
homomorphism.



Note finally that there is a natural homomorphism $\MCG(S) \to
\Aut(\calC(S))$ (and similarly for $V$).  We will call any element of
the image of this map a {\em geometric automorphism}.  Our main
theorem is:

\begin{restate}{Theorem}{Thm:AutOfDiskComplex}
If a handlebody $V = V_{g,n}$ satisfies $e(V) \geq 3$ then
the natural map $\calH(V) \to \Aut(\calD(V))$ is a surjection.
\end{restate}

In the language above: every element of $\calH(V)$ is geometric.  The
plan of the proof of \refthm{AutOfDiskComplex} is given in
\refsec{Sketch} and completed in \refsec{CrawlDisk}.  \refsec{Small}
shows that \refthm{AutOfDiskComplex} is sharp; all handlebodies $V$
with $e(V) \leq 2$ exhibit some kind of exceptional behaviour.
\refthm{AutOfDiskComplex} has a corollary:

\begin{restate}{Theorem}{Thm:AutOfDiskComplexII}
If a handlebody $V = V_{g,n}$ satisfies $e(V) \geq 3$ then
the natural map $\calH(V) \to \Aut(\calD(V))$ is an isomorphism.
\end{restate}

In \refsec{AutOfHandlebodyGroup} we use \refthm{AutOfDiskComplex} to
prove:

\begin{restate}{Theorem}{Thm:AutOfHandlebodyGroup}
If $e(V) \geq 3$ then the outer automorphism group of the
handlebody group is trivial.
\end{restate}

These results are inspired by work of Ivanov, Korkmaz and
Luo~\cite{Ivanov97, Korkmaz99, Luo00}:

\begin{theorem}
\label{Thm:Ivanov}
If $3g - 3 + n \geq 3$, or if $(g, n) = (0, 5)$, then all elements of
$\Aut(\calC(S_{g,n}))$ are geometric.  Also, the outer automorphism
group of $\MCG(S)$ is trivial.  \qed
\end{theorem}

\section{Background}
\label{Sec:Background}

The genus zero case of \refthm{Ivanov} is contained in the thesis of
the first author~\cite[Theorem 1]{Korkmaz99}.

\begin{theorem}
\label{Thm:Korkmaz}
If $g = 0$ and $n \geq 5$ then all elements of $\Aut(\calC(S_{0,n}))$
are geometric. \qed
\end{theorem}

Spotted balls are the simplest handlebodies.  Accordingly:

\begin{lemma}
\label{Lem:SpottedBalls}
The natural maps $\calD(V_{0,n}) \to \calC(S_{0,n})$ and
$\calH(V_{0,n}) \to \MCG(S_{0,n})$ are isomorphisms.
\end{lemma}

\begin{proof}
The three-manifold $V_{0,n}$ is an $n$--spotted ball.  Every simple
closed curve in $\bdy_+ V$ bounds a disk in $V$.  This proves that
$\calD(V_{0,n}) \to \calC(S_{0,n})$ is a surjection and thus, by the
remark immediately after \refdef{DiskComplex}, an isomorphism.

It follows from the Alexander trick that the inclusion of mapping
class groups is an isomorphism.
\end{proof}

The genus zero case of \refthm{AutOfDiskComplex} is an immediate
corollary.  We now give basic definitions.

Suppose that $V$ is a handlebody.  Two disks $D, E \in \calD(V)$ are
{\em topologically equivalent} if there is a mapping class $f \in
\calH(V)$ so that $f(D) = E$.  The {\em topological type} of $D$ is
its equivalence class in $\calD(V)$.

For any simplicial complex, $\calK$, if $\sigma \in \calK$ is a
simplex then recall that
$$ 
\link(\sigma) = \{ \tau \in \calK \st \sigma \cap \tau =
\emptyset,~\sigma \cup \tau \in \calK \}.
$$ 
So if $\DD$ is a simplex of $\calD(V)$ then $\link(\DD)$ is the
subcomplex of $\calD(V)$ spanned by disks $E$ disjoint from some $D
\in \DD$ and distinct from all $D \in \DD$.  


If $X \subset Y$ is a properly embedded submanifold then we write
$\neigh(X)$ and $\Neigh(X)$ to denote open and closed regular
neighborhoods of $X$ in $Y$.  If $X$ is codimension zero then the {\em
frontier} of $X$ in $Y$ is the closure of $\bdy X \setminus \bdy Y$.  

A simplex $\DD \in \calD(V)$ is a {\em cut system} if $V \setminus
\neigh(\DD)$ is a spotted ball.  Note that every disk of $\DD$ yields
two spots of $V \setminus \neigh(\DD)$. 

Recall that for simple curves $\alpha, \beta$ properly embedded in $S$
the {\em geometric intersection number} $i(\alpha, \beta)$ is the
minimum possible intersection number between proper isotopy
representatives.

Two disks $D, E \in \calD(V)$ are {\em dual} if $i(\bdy D, \bdy E) =
2$; equivalently, after a suitable proper isotopy $D$ and $E$
intersect along a single arc; equivalently, after a suitable proper
isotopy a regular neighborhood of $D \cup E$ is a four-spotted ball
with all spots essential in $V$.  See \reffig{Dual}.

\begin{figure}[htbp]
$$\begin{array}{c}
\epsfig{file=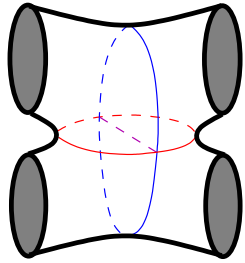, height = 3.5 cm}
\end{array}$$
\caption{Every spot of the $V_{0,4}$ containing a pair of dual disks
is essential in $V$.}
\label{Fig:Dual}
\end{figure}

If $\DD = \{ D_i \}$ is a cut system we define $\dual_i(\DD)$ to be
the subcomplex spanned by the disks $E \in \calD(V)$ which are dual to
$D_i$ and disjoint from $D_j$ for all $j \neq i$.  We take
$\dual(\DD)$ to be the complex spanned by $\cup_i \dual_i(\DD)$.

\section{The proof of \refthm{AutOfDiskComplex}}
\label{Sec:Sketch}



Let $V = V_{g,n}$ be a genus $g$ handlebody with $n$ spots.  We
suppose that $g \geq 1$ and $e(V) \geq 3$.  Let $\phi$ be any
automorphism of $\calD(V)$.  \reflem{Types} proves that $\phi$
preserves the topological types of disks.  In addition, $\phi$ sends
cut systems to cut systems (\refclm{Cut}).  Next \reflem{Dual} shows
that $\phi$ preserves duality.  Also, for any cut system $\DD = \{ D_i
\}$, the complex $\dual_i(\DD)$ is connected
(\reflem{DualComplexConnected}).

Pick any geometric automorphism $f_{\cut}$ so that $f_{\cut}(\DD) =
\phi(\DD)$, vertex-wise; $f_{\cut}$ exists by \refclm{Cut}.  Define
$\phi_{\cut} = f^{-1}_{\cut} \circ \phi$.  Thus
$$
\phi_{\cut} | \DD = \Id.
$$

Let $V' \isom V_{0, 2g + n}$ be the spotted ball obtained by cutting
$V$ along a regular neighborhood of $\DD$.  Now, since $\phi_{\cut}$
preserves $\link(\DD) \isom \calD(V')$, by \refthm{Korkmaz} and
\reflem{SpottedBalls} there is a homeomorphism $f \from V' \to V'$ so
that the induced automorphism $f \in \Aut(\calD(V'))$ satisfies $f =
\phi_{\cut} | \link(\DD)$.  \refsec{Reglue} proves that $f$ preserves
the $g$ pairs of spots of $V'$ coming from $\DD$.  Thus $f$ can be
glued to give a homeomorphism $f_{\link} \from V \to V$ as well as an
induced geometric automorphism $f_{\link} \in \Aut(\calD(V))$.  Define
$\phi_{\link} = f^{-1}_{\link} \circ \phi_{\cut}$.  Thus
$$
\phi_{\link} | \DD \cup \link(\DD) = \Id.
$$

Recall that $\phi_{\link}$ preserves duals by \reflem{Dual}.  For
every $D_i \in \DD$ pick some dual $E_i \in \dual_i(\DD)$.  By
\reflem{PreCrawl} there is an integer $m_i \in \ZZ$ so that
$T_i^{m_i}(E_i) = \phi_{\link}(E_i)$, where $T_i$ is the Dehn twist
about $D_i$.  Define $f_{\dual} = \prod T_i^{m_i}$ and define
$\phi_{\dual} = f_{\dual}^{-1} \circ \phi_{\link}$.  Letting $\EE = \{
E_i \}$ we have
\[
\phi_{\dual} | \DD \cup \link(\DD) \cup \EE = \Id.
\]

Recall that \reflem{DualComplexConnected} proves that $\dual_i(\DD)$
is connected.  Therefore, a {\em crawling} argument, given in
\reflem{CrawlDual}, proves that
$$
\phi_{\dual} | \DD \cup \link(\DD) \cup \dual(\DD) = \Id.
$$ 
Wajnryb~\cite{Wajnryb98} proves that the {\em cut system complex} is
connected. Thus we may likewise crawl through $\calD(V)$ and prove
(\refsec{CrawlDisk}) that
$$
\phi_{\dual} = \Id
$$
and so prove that 
$$
\phi = f_{\cut} \circ f_{\link} \circ f_{\dual}.
$$
Thus $\phi$ is geometric.

\section{Small handlebodies}
\label{Sec:Small}


In this section we deal with the small cases, where $e(V) =
2g - 2 + n \leq 2$.  We start with genus zero.  If $n \leq 3$ then
$\calD(V_{0,n})$ is empty.  By \reflem{SpottedBalls} the mapping class
groups of $V$ and $\bdy_+ V$ are equal.  Thus
\[
\calH(V_{0}),~\calH(V_{0, 1}) \isom \ZZ/2\ZZ
\] 
while 
\[
\calH(V_{0, 2}) \isom K_4 \quad \rm{and} \quad \calH(V_{0, 3})
\isom \ZZ/2\ZZ \cross \Sigma_3.
\]
Here $K_4$ is the Klein four-group and $\Sigma_3$ is the symmetric
group on three objects~\cite[Appendix A]{RafiSchleimer08}.

If $n = 4$ then $\calD$ is a countable collection of vertices with no
higher dimensional simplices.  Thus $\Aut(\calD) = \Sigma_\infty$ is
uncountable.  However, there are only countably many geometric
automorphisms.  In fact, by \reflem{SpottedBalls}, the mapping class
group $\calH(V_{0,4})$ is isomorphic to $K_4 \rtimes \PGL(2,
\ZZ)$~\cite[Appendix A]{RafiSchleimer08}.

For genus one, if $n = 0$ or $1$ then $\calD$ is a single point and
$\Aut(\calD)$ is trivial.  On the other hand
$$
\calH(V_{1}),~\calH(V_{1, 1}) \isom \ZZ \rtimes K_4.
$$ 
For $V = V_{1,2}$ matters are more subtle.  The subcomplex $\NonSep(V)
\subset \calD(V)$, spanned by non-separating disks, is a copy of the
Bass-Serre tree for the meridian curve in $S_{1,1} = \bdy_+
V_{1,1}$~\cite{KentEtAl06}.
Thus $\NonSep(V)$ is a copy of $T_\infty$: the regular tree with
countably infinite valance.  Now, if $E \in \calD(V)$ is separating
then there is a unique disk $D$ disjoint from $E$; also, $D$ is
necessarily non-separating.  It follows that $\calD(V)$ is a copy of
$\NonSep(V)$ with countably many leaves attached to every vertex.
Thus $\Aut(\calD)$ contains a copy of $\Aut(T_\infty)$ as well as
countably many copies of $\Sigma_\infty$ and is therefore uncountable.
As usual $\calH(V)$ is countable and so $\Aut(\calD)$ contains
non-geometric elements.  However, following Luo's treatment of
$\calC(S)$~\cite{Luo00} suggests the following problem:

\begin{problem}
\label{Prob:Dual}
Suppose that $V = V_{1,2}$.  Let $\calG$ be the subgroup of
$\Aut(\calD(V))$ consisting of automorphisms preserving duality: if
$\phi \in \calH$ and $D, E$ are dual then so are $\phi(D), \phi(E)$.
Is every element of $\calG$ geometric?
\end{problem}

Note that this approach of recording duality is precisely correct for
the four-spotted ball; the complex where simplices record duality in
$V_{0,4}$ is the {\em Farey tessellation}, $\calF$, and every element
of $\Aut(\calF)$ is geometric.  See~\cite[Section 3.2]{Luo00}.

The last exceptional case is $V = V_2$.  Let $\NonSep(V)$ be the
subcomplex of $\calD(V)$ spanned by non-separating disks.  Then
$\NonSep(V)$ is an increasing union, as follows: 
$\calN_0$ is a single triangle, $\calN_{i+1}$ is obtained by attaching
(to every free edge of $\calN_i$) a countable collection of triangles,
and $\NonSep(V)$ is the increasing union of the $\calN_i$.  A careful
discussion of $\NonSep(V)$ is given by Cho and
McCullough~\cite[Section 4]{ChoMcCullough06}

We obtain $\calD(V)$ by attaching a countable collection of triangles
to every edge of $\NonSep(V)$.  To see this note that every separating
disk $E$ divides $V$ into two copies of $V_{1,1}$.  These copies of
$V_{1,1}$ have meridian disks, say $D$ and $D'$.  Thus $\link(E)$ is
an edge and the triangle $\{ E, D, D' \}$ has two free edges in
$\calD(V)$, as indicated.  Finally, there is a countable collection of
separating disks lying in $V \setminus (D \cup D')$, again as
indicated.

It follows that $\Aut(\calD(V_2))$ is uncountable.  Again, as in
\refprob{Dual}, we may ask: are all ``duality-respecting'' elements $f
\in \Aut(\calD(V_2))$ geometric?  We end with another open problem:

\begin{problem}
\label{Prob:NonSep}
Suppose that $V$ is a handlebody with $e(V)$ and genus both
sufficiently large. Show that $\Aut(\NonSep(V)) = \calH(V)$.
\end{problem}


A solution to \refprob{NonSep} may lead to a simplified proof of
\refthm{AutOfHandlebodyGroup}.

\section{Topological types}
\label{Sec:Types}

The goal of this section is:

\begin{lemma}
\label{Lem:Types}
Suppose that $\phi \in \Aut(\calD(V))$.  Then $\phi$ preserves
topological types of disks.
\end{lemma}

The {\em complexity} of $V_{g,n}$ is $\xi(V) = 3g - 3 + n$.  If
$\xi(V) \geq 1$ then $\xi(V)$ is the number of vertices of a maximal
simplex of $\calD(V)$.  Note that $V_1$, $V_{1,1}$ and $V_{0,4}$ are
the only handlebodies where $\calD(V)$ has dimension zero.  (When
$\calD(V)$ is empty its dimension is $-1$.)  Further $V_1$ and
$V_{1,1}$ are the only handlebodies where $\calD(V)$ is a single
point.

We will call $V_{0,3}$, the three-spotted ball, a {\em solid pair of
pants}.  Thus $\xi(V)$ is the number of disks in a {\em pants
decomposition} of $V$ while $e(V) = 2g - 2 + n$ is the
number of solid pants in the decomposition.  We will call $V_{1,1}$ a
{\em solid handle}.  Suppose now that $E$ is separating with $V
\setminus \neigh(D) = X \cup Y$.  If $X$ or $Y$ is a solid pants then
we call $E$ a {\em pants disk}.  If $X$ or $Y$ is a solid handle then
we call $E$ a {\em handle disk}.

Recall that if $\calK$ and $\calL$ are non-empty simplicial complexes
with disjoint vertex sets then $\calK \join \calL$, their {\em join},
is the complex
$$ 
\calK \cup \{ \sigma \cup \tau \st \sigma \in \calK,~\tau \in \calL \}
\cup \calL.
$$ 

\begin{claim}
\label{Clm:NotJoin}
For any handlebody $V$ the complex $\calD(V)$ is not a join.
\end{claim}

\begin{proof}
When $e(V) \leq 2$ this can be checked case-by-case, following
\refsec{Small}.  The remaining handlebodies all admit disks $D, E$
that {\em fill}: every disk $F$ meets at least one of $D$ or $E$.  It
follows that any edge-path in $\calD^{(1)}(V)$ connecting $D$ to $E$
has length at least three.  However, the diameter of the one-skeleton
of a join is either one or two.
\end{proof}

The complex $\calD(V)$ is {\em flag}: minimal non-faces have dimension
one.

Observe that $\phi$ preserves the combinatorics of $\calD(V)$.  Thus
any topological property of $V$ that has a combinatorial
characterization will be preserved by $\phi$.  We proceed with a
sequence of claims.

\begin{claim}
\label{Clm:Join}
The disk $E$ is a separating disk yet not a pants disk if and only if
$\link(E)$ is a join.  Furthermore, in this case $\link(E)$ is
realized as a join in exactly one way, up to permuting the factors.
\end{claim}

\begin{proof}
Suppose that $V \setminus \neigh(E) = X \cup Y$, where neither $X$ nor
$Y$ is a solid pants.  Since $E$ is essential and non-peripheral both
$\calD(X)$ and $\calD(Y)$ are non-empty.  It follows that $\link(E) =
\calD(X) \join \calD(Y)$, and neither factor is empty.  Furthermore,
this join is realized uniquely, because $\calD(X)$ is never itself a
join (by \refclm{NotJoin}), $\calD(X)$ is flag and join is
associative.



On the other hand, if $E$ is non-separating then $\link(E)$ is
isomorphic to $\calD(V_{g - 1, n + 2})$.  If $E$ is a pants disk then
$\link(E) \isom \calD(V_{g, n - 1})$.  Neither of these is a join by
\refclm{NotJoin}. 
\end{proof}

A {\em cone} is the join of a point with some non-empty simplicial
complex.

\begin{claim}
\label{Clm:Handle}
Suppose that $V \neq V_{1,2}$.  Then $E \in \calD(V)$ is a handle disk
if and only if $\link(E)$ is a cone.
\end{claim}

\begin{proof}
Suppose that $E$ cuts off a solid handle $X$ with meridian $D$.  Let
$Y$ be the other component of $V \setminus \neigh(E)$.  Since $V \neq
V_{1,2}$ we have that $\calD(Y)$ is non-empty; in particular $E$ is
not a pants disk.  By \refclm{Join} we have $\link(E) = \calD(X) \join
\calD(Y)$.  As $\calD(X) = \{ D \}$ we are done with the forward
direction.

Now suppose that $\link(E)$ is a cone from $D$.  Since a cone is the
join of the apex with the base, by \refclm{Join} the disk $E$ is
separating.  Let $V \setminus \neigh(E) = X \cup Y$.  Thus $\link(E) =
\calD(X) \join \calD(Y)$.  However, by \refclm{Join} the decomposition
of $\link(E)$ is unique; breaking symmetry we may assume that
$\calD(X) = \{ D \}$.  Thus $X$ is a solid handle and we are done.
\end{proof}

It immediately follows that: 

\begin{claim}
\label{Clm:NonSep}
Suppose that $V \neq V_{1,2}$.  Then $D \in \calD(V)$ is
non-separating if and only if there is an $E \in \calD(V)$ so that
$\link(E)$ is a cone with apex $D$. \qed
\end{claim}

\begin{claim}
\label{Clm:Cut}
Suppose that $e(V) \geq 3$.  A simplex $\DD \in \calD(V)$ is a
cut system if and only if the following properties hold:
\begin{itemize}
\item
for every pair of disks $D, E \in \link(\DD)$ the complex $\link(E)
\cap \link(\DD)$ is not a cone with apex $D$ and
\item
for every proper subset $\sigma \subsetneq \DD$ there is a pair of
disks $D, E \in \link(\sigma)$ so that the complex $\link(E) \cap
\link(\sigma)$ is a cone with apex $D$.
\end{itemize}
\end{claim}

\begin{proof}
The forward direction follows from \refclm{NonSep} and the definition
of a cut system.  (When $V$ is a spotted ball the only cut system is
the empty set; the empty set has no proper subsets.)

Now for the backwards direction: From the first property and by
\refclm{NonSep} deduce that $V' = V \setminus \neigh(\DD)$ is a
collection of spotted balls.  If $V'$ has at least two components then
there is a proper subset $\sigma \subset \DD$ which is a cut system
for $V$.  Thus $V \setminus \neigh(\sigma)$ is a spotted ball and this
contradicts the second property.
\end{proof}

\begin{lemma}
\label{Lem:Characterize}
Suppose that $V, W$ are handlebodies with $\calD(V) \isom \calD(W)$.
Then either:
\begin{itemize}
\item
$V \homeo W$ or
\item
$V, W \in \{ V_1, V_{1,1} \}$ or
\item
$V, W \in \{ V_0, V_{0,1}, V_{0,2}, V_{0,3} \}$.
\end{itemize}
\end{lemma}

This is the handlebody version of~\cite[Lemma 4.5]{Korkmaz99}
and~\cite[Lemma 2.1]{Luo00}.

\begin{proof}[Proof of~\reflem{Characterize}]
When $e(V) \leq 2$ this can be checked case-by-case, following
\refsec{Small}.  When $V$ has $e(V) \geq 3$ then $\xi(V) = \xi(W)$.
By \refclm{Cut} the handlebodies $V$ and $W$ have cut systems of the
same size.  It follows that $V, W$ have the same genus and thus the
same number of spots.
\end{proof}

We now have: 

\begin{proof}[Proof of \reflem{Types}]
Let $V = V_{g,n}$ and fix $\phi \in \Aut(\calD(V))$.  When $e(V) \leq
2$, \reflem{Types} can be checked case-by-case, following
\refsec{Small}.  So suppose that $e(V) \geq 3$.

The automorphism $\phi$ must preserve the set of non-separating disks
by \refclm{NonSep}.

Suppose that $E \in \calD(V)$ is a separating disk yet not a pants
disk.  Writing $V \setminus \neigh(E) = X \cup Y$ we have $\link(E) =
\calD(X) \join \calD(Y)$.  By \refclm{Join} this join is realized
uniquely and so we can recover $\calD(X)$ and $\calD(Y)$.  By
\reflem{Characterize} we may deduce, combinatorially, the genus and
number of spots of $X$ and $Y$.  Thus $\phi$ preserves the topological
type of $E$.

The only topological type remaining is the set of pants disks.  Since
all other types are preserved, so are the pants disks.  We are done.
\end{proof}

\section{Regluing}
\label{Sec:Reglue}

Suppose that $\phi_\DD \in \Aut(\calD(V))$ fixes $\DD$.  By
\reflem{SpottedBalls} there is a homeomorphism $f$ of $V' = V
\setminus \neigh(\DD)$ so that the induced geometric automorphism
equals $\phi_\DD|\link(\DD)$.  We must show that $f$ gives a
homeomorphism of $V$: that is, for every $i$ the spots $D_i^\pm$ are
preserved by $f$.

Let $\handle_i(\DD) \subset \link(\DD)$ be the collection of handle
disks $E \in \calD(V)$ such that
\begin{itemize}
\item
one component of $V \setminus \neigh(E)$ is a solid handle containing
$D_i$ and
\item
$E$ is disjoint from all of the $D_j$.
\end{itemize}

Let $\pants_i(\DD) \subset \calD(V')$ be the collection of pants disks
$E$ such that one component of $V' \setminus \neigh(E)$ is a solid
pants meeting the spots $D_i^\pm$.

By the claims in the previous section the set $\handle_i(\DD)$ is, for
all $i$, combinatorially characterized and so preserved by
$\phi_{\DD}$.  It follows that the homeomorphism $f \in \Homeo(V')$
preserves the set $\pants_i(\DD)$, for all $i$.  Now, suppose that
$f(D_1^+), f(D_1^-) = A, B$ where $A, B$ are spots of $V'$.  Let $E
\in \pants_1(\DD)$ be any pants disk.  Then $f(E)$ is a pants disk
cutting off $A$ and $B$.  It follows that the spots $A, B$ (in some
order) equal the spots $D_1^\pm$ as desired.

\section{Duality}
\label{Sec:Dual}

Recall that two disks $D, E \in \calD(V)$ are {\em dual} if $i(\bdy D,
\bdy E) = 2$ (see \reffig{Dual}).  A {\em pentagon} $P \subset
\calD(V_{0,5})$ is a collection of five disks $P = \{ E_i \}_{i =
0}^4$ so that $E_i$ and $E_{i+1}$ are disjoint, for all $i$ (modulo
five).  We say that the disks $E_i, E_{i+2}$ are {\em non-adjacent} in
$P$, for all $i$ (modulo five).

\begin{lemma}[Pentagon Lemma]
\label{Lem:Pentagon}
Suppose that $V = V_{0,5}$.  Two disks $D, E \in \calD(V)$ are dual if
and only if there is a pentagon $P$ so that $D, E \in P$ and $D, E$
are non-adjacent in $P$.  
\end{lemma}

\begin{proof}
Recall that $\calD(V_{0,5}) \isom \calC(S_{0,5})$, by
\reflem{SpottedBalls}.  The {\em pentagon lemma} for $S_{0,5}$
(see~\cite[Theorem 3.2]{Korkmaz99} or~\cite[Lemma 4.2]{Luo00}) implies
that there is only one pentagon in $\calD(V_{0,5})$, up to the action
of the handlebody group.
\end{proof}

\begin{lemma}
\label{Lem:Dual}
Suppose that $V = V_{g,n}$ has $e(V) \geq 3$.
Two disks $D, E \in \calD(V)$ are dual if and only if there is a
simplex $\sigma \in \calD(V)$ with
\begin{itemize}
\item
$\link(\sigma) \isom \calD(V_{0,5})$, 
\item
$D, E$ are non-adjacent in some pentagon of $\link(\sigma)$. 
\end{itemize}
\end{lemma}

It follows that every $\phi \in \Aut(\calD(V))$ preserves duality.  We
will say that a handlebody $W \subset V$ is {\em cleanly embedded} if:
\begin{itemize}
\item
all spots of $W$ are essential in $V$ and
\item
if a spot of $W$ is peripheral in $V$ then it is also a spot of $V$.
\end{itemize}

\begin{proof}[Proof of \reflem{Dual}]
Suppose that $D, E$ are dual.  Let $X$ be the four-spotted ball
containing them.  Isotope $X$ to be cleanly embedded.  Let $\EE$ be a
pants decomposition of $V' = V \setminus \neigh(X)$.  Now, there is at
least one solid pants $P$ in $V' \setminus \neigh(\EE)$ which has a
spot, say $F$, which is parallel to a spot of $X$.  If not then $e(V)
\leq 2$, a contradiction.

Let $Y = X \cup \neigh(F) \cup P$ and notice that this is a
five-spotted ball containing $D$ and $E$, our original disks.  Isotope
$Y$ to be cleanly embedded.  Let $\EE'$ be any pants decomposition of
$V \setminus \neigh(Y)$.  Add to $\EE'$ any spots of $Y$ which are
non-peripheral in $V$.  This then is the desired simplex $\sigma \in
\calD(V)$.  Since $D$ and $E$ are dual the pentagon lemma implies that
there is a pentagon in $\calD(Y)$ making $D, E$ non-adjacent.

The backwards direction follows from \reflem{Characterize}, the
combinatorial characterization of genus and number of spots, and from
the pentagon lemma.
\end{proof}

We now discuss the dual complex.  Fix a cut system $\DD = \{ D_i \}$.
Recall that $\dual_i(\DD)$ is the subcomplex of $\calD(V)$ spanned by
the disks $E \in \calD(V)$ which are dual to $D_i$ and disjoint from
$D_j$ for all $j \neq i$.

Define $V_i$ to be the spotted solid torus obtained by cutting $V$
along all disks of $\DD$ {\em except} $D_i$.  Note that $V_i$ has
exactly $e(V)$--many spots, and this is at least three.  Also, $D_i$
is a meridian disk for $V_i$.  Note that $\dual_i(\DD) \subset
\calD(V_i)$.  A disk $E \in \dual_i(\DD)$ is a {\em simple dual} if
$E$ is a pants disk in $V_i$.


Let $\calA_i(\DD)$ be the complex where vertices are isotopy classes
of arcs $\alpha \subset \bdy_+ V_i$ so that
\begin{itemize}
\item
$\alpha$ meets $\bdy D_i$ exactly once, transversely, and
\item
$\bdy \alpha$ meets distinct spots of $V_i$. 
\end{itemize}
A collection of vertices spans a simplex if they can be realized
disjointly.  

If an arc $\alpha \in \calA_i(\DD)$ meets spots $A, B \in \bdy_0 V_i$
then the frontier of $\Neigh(A \cup \alpha \cup B)$ is a simple dual,
$E_\alpha$.  



\begin{lemma}
\label{Lem:DualComplexConnected}
If $e(V) \geq 3$ then the complex $\dual_i(\DD)$ is connected.
\end{lemma}

It suffices to check this for $i = 1$.  To simplify notation we write
$D = D_1$, $U = V_1$, $\dual(D) = \dual_i(\DD)$ and $\calA(D) =
\calA_1(\DD)$.  We will prove \reflem{DualComplexConnected} via a
sequence of claims.

\begin{claim*}
For any pair of arcs $\alpha, \gamma \in \calA(D)$ there is a sequence
$\{ \alpha_k \}_{k = 0}^N \subset \calA(D)$ so that:
\begin{itemize}
\item
the arcs $\alpha_k, \alpha_{k+1}$ are disjoint, for all $k < N$, 
\item
$\alpha_0 = \alpha$ and $\alpha_N = \gamma$, and 
\item
there is at most one spot in common between the endpoints of
$\alpha_k$ and $\alpha_{k+1}$, for all $k < N$. 
\end{itemize}
\end{claim*}

\begin{proof}
Fix, for the remainder of the proof, an arc $\beta \in \calA(D)$ so
that $\alpha$ and $\beta$ are disjoint and so that the endpoints of
$\alpha$ and $\beta$ share at most one spot.  This is possible as $U$
has at least three spots.  Define the complexity of $\gamma$ to be
$c(\gamma) = i(\alpha, \gamma) + i(\beta, \gamma)$.  Notice if
$c(\gamma) = 0$ then we are done: one of the sequences
$$ 
\{ \alpha, \gamma \} \quad \mbox{or} \quad \{ \alpha, \beta, \gamma \}
$$
has the desired properties.

Now induct on $c(\gamma)$.  Suppose, breaking symmetry, that $\alpha$
meets a spot, say $A \in \bdy_0 U$, so that $\gamma \cap A =
\emptyset$.  If $i(\alpha, \gamma) = 0$ then the sequence $\{ \alpha,
\gamma \}$ has the desired properties.  If not, then let $x$ be the
point of $\alpha \cap \gamma$ that is closest, along $\alpha$, to the
endpoint $\alpha \cap A$.  Let $\alpha' \subset \alpha$ be the subarc
connecting $x$ and $\alpha \cap A$.  Let $N$ be a regular
neighborhood, taken in $\bdy_+ U$, of $\gamma \cup \alpha'$.  The
frontier of $N$, in $\bdy_+ U$, is a union of three arcs: one arc
properly isotopic to $\gamma$ and two more arcs $\gamma', \gamma''$.

The arcs $\gamma'$ and $\gamma''$ are disjoint from $\gamma$ and
satisfy $c(\gamma') + c(\gamma'') \leq c(\gamma) - 1$.  Also, since
$\gamma'$ and $\gamma''$ each have one endpoint on the spot $A$ the
arcs $\gamma'$ and $\gamma''$ have exactly one spot in common with
$\gamma$.  Now, if $\alpha' \cap \bdy D = \emptyset$ then one of
$\gamma', \gamma''$ meets $\bdy D$ once and the other is disjoint.  On
the other hand, if $\alpha' \cap \bdy D \neq \emptyset$ then $\alpha'$
meets $\bdy D$ once.  Thus one of $\gamma', \gamma''$ meets $\bdy D$
once and the other meets $\bdy D$ twice.  In either case we are done.
\end{proof}

Recall that if $\alpha \in \calA(D)$ is an arc then $E_\alpha$ is the
associated simple dual.

\begin{claim*}
If $\alpha, \beta \in \calA(D)$ are disjoint arcs, with at most one
spot in common between their endpoints, then there is an edge-path in
$\dual(D)$ of length at most four between $E_\alpha$ and $E_\beta$.
\end{claim*}

\begin{proof}
If $\alpha$ and $\beta$ share no spots then $\{ E_\alpha, E_\beta \}$
is a path of length one.  Suppose that $\alpha$ and $\beta$ share a
single spot.  Let $A, B, C$ be the three spots that $\alpha$ and
$\beta$ meet, with both meeting $C$.  Let $\alpha', \beta'$ be the
subarcs of $\alpha, \beta$ connecting $C$ to $\bdy D$.  There are two
cases: either $\alpha'$ and $\beta'$ are incident on the same side of
$\bdy D$ or are incident on opposite sides.

Suppose that $\alpha'$ and $\beta'$ are incident on the same side of
$\bdy D$.  Then $\alpha'$ and $\beta'$, together with subarcs of $\bdy
C$ and $\bdy D$ bound a disk $\Delta \subset \bdy U$.  Note that
$\Delta$ may contain spots, but it meets $A \cup B \cup C$ only along
the subarc in $\bdy C$.  It follows that the disk $F$, defined to be
the frontier of
$$
\Neigh\left( (A \cup B \cup C) \cup (\alpha \cup \beta) \cup \Delta
\right),
$$ 
is dual to $D$.  The disk $F$ is also essential as it separates at
least three spots from a solid handle.  So $\{ E_\alpha, F, E_\beta
\}$ is the desired path.

Suppose that $\alpha'$ and $\beta'$ are incident on opposite sides of
$\bdy D$.  Let $d \subset \bdy D$ be either component of $\bdy D
\setminus (\alpha \cup \beta)$.  Let $\alpha'' = \closure{\alpha
\setminus \alpha'}$ and define $\beta''$ similarly.  Define $\gamma
\in \calA(D)$ by forming the arc $\alpha'' \cup d \cup \beta''$ and
using an proper isotopy of $\bdy_+ U$ to make $\gamma$ transverse to
$\bdy D$.  Now apply the previous paragraph to the pairs $\{ \alpha,
\gamma \}$ and $\{ \gamma, \beta \}$ to obtain the desired path of
length four.
\end{proof}

\begin{claim*}
For every dual $E \in \dual(D)$ there is a simple dual connected to
$E$ by an edge-path of length at most two.
\end{claim*}

\begin{proof}
The graph $\bdy E \cup \bdy D$ cuts $\bdy U$ into a pair of disks $B,
C$ and an annulus $A$.  Each of $B, C$ contain at least one spot.

Suppose $E$ is separating.  Then the disks $B, C$ are adjacent along
an subarc of $\bdy D$.  Connect a spot in $B$ to a spot in $C$ by an
arc $\alpha$ that meets $\bdy D$ once and that is disjoint from
$\bdy E$.  Thus $E_\alpha$ is disjoint from $E$.

Suppose $E$ is non-separating.  Then the two disks $B, C$ meet only at
the points of $\bdy D \cap \bdy E$.  Now, if the annulus $A$ contains
a spot then we may connect a spot in $B$ to a spot in $A$ by an arc
$\alpha$ meeting $\bdy D$ once and $\bdy E$ not at all.  In this case
we are done as in the previous paragraph.

If $A$ contains no spots then, breaking symmetry, we may assume
that $B$ contains at least two spots while $C$ contains at least one.
Let $\delta$ be an arc connecting some spot, say $B' \subset B$, to
$E$.  Let $N$ be a regular neighorhood of $E \cup \delta \cup B'$.
Then the frontier of $N$ contains two disks.  One of these is isotopic
to $E$ while the other, say $E'$, is non-separating, dual to $D$, and
divides the spots as described in the previous paragraph.
\end{proof}

Equipped with these claims we have:

\begin{proof}[Proof of \reflem{DualComplexConnected}]
The first two claims imply that the set of simple duals in $\dual(D)$
is contained in a connected set.  The third claim shows that every
vertex in $\dual(D)$ is distance at most two from the set of simple
duals.  This completes the proof.
\end{proof}

\section{Crawling through the complex of duals}
\label{Sec:CrawlDual}

\begin{lemma}
\label{Lem:PreCrawl}
Suppose that $\phi_{\link}$ fixes $\DD$ and $\link(\DD)$.  For any $E
\in \dual_i(\DD)$ the disks $E$ and $\phi(E)$ differ by some power of
$T_i$, the Dehn twist about $D_i$.
\end{lemma}

\begin{proof}
As usual, it suffices to prove this for $D = D_1$.  Let $U = V_1$.

Let $X \subset U$ be the four-spotted ball filled by $D$ and the dual
disk $E$.  Isotope $X$ to be cleanly embedded.  Let $\FF$ be the
components of $\bdy_0 X$ which are not spots of $U$.  Note that
$\phi_{\link}$ fixes $D$ as well as every disk of $\FF$.  This,
together with \reflem{Dual}, implies that $\phi_{\link}$ preserves the
set of disks that are contained in $X$ and dual to $D$.

Since $\calD(X)$ equipped with the duality relation is a copy of
$\calF$, the Farey graph, it follows that $E$ and $F =
\phi_{\link}(E)$ differ by some number of half-twists about $D$.  If
$E$ and $F$ differ by an odd number of half-twists then $E$ and $F$
have differing topological types, contradicting \reflem{Types} applied
to $\phi_{\link}|\calD(U)$.  Thus $E$ and $F$ differ by an even number
of half-twists, as desired.
\end{proof}

\begin{lemma}
\label{Lem:CrawlDual}
Suppose that $\phi_{\dual}$ fixes $\DD$, $\link(\DD)$, and $\EE$, a
collection of duals (that is, $E_i \in \dual_i(\DD)$).  Then
$\phi_{\dual}$ fixes every vertex of $\dual_i(\DD)$, for all $i$.
\end{lemma}


\begin{proof}
As usual, it suffices to prove this for $D = D_1$.  Let $E = E_1$ and
let $U = V_1$.  We {\em crawl} through $\dual(D) = \dual_1(\DD)$, as
follows.  

Suppose that $F, G \in \dual(D)$ are adjacent vertices and suppose
that $\phi_{\dual}(F) = F$.  By \reflem{PreCrawl}, the disks $G$ and
$G' = \phi_{\dual}(G)$ differ by some number of Dehn twists about $D$.
Also, as $\phi_{\dual}$ is a simplical automorphism the disks $F$ and
$G'$ are disjoint.  Let $X$ be the four-spotted ball filled by $D$ and
$F$.  If $G$ and $G'$ are not equal then $G \cap X$ and $G' \cap X$
are also not equal and in fact differ by some non-zero number of
twists; thus one of $G \cap X$ or $G' \cap X$ must cross $F$, a
contradiction.

Recall that $\phi_{\dual}(E) = E$.  Suppose that $G$ is any vertex of
$\dual(D)$.  Since $\dual(D)$ is connected
(\reflem{DualComplexConnected}) there is a path $\calP \subset
\dual(D)$ connecting $E$ to $G$.  Induction along $\calP$ completes
the proof.
\end{proof}

\section{Crawling through the disk complex}
\label{Sec:CrawlDisk}

Before continuing we will need the following complex:

\begin{definition}[Wajnryb~\cite{Wajnryb98}]
\label{Def:CutGraph}
The {\em cut system graph} $\CG(V)$ is the graph with vertex set being
isotopy classes of unordered cut systems in $V$.  Edges are given by
pairs of cut systems with $g - 1$ disks in common and the remaining
pair of disks disjoint.
\end{definition}

Wajnryb also gives a two-skeleton, but we will only require:

\begin{theorem}[Wajnryb~\cite{Wajnryb98}]
\label{Thm:Wajnryb}
The cut system graph $\CG(V)$ is connected.
\end{theorem}

For the remainder of this section suppose that $\Phi = \phi_{\dual}$
is an automorphism of $\calD(V)$ and $\DD$ is a cut system so that
$\Phi$ fixes $\DD$, $\link(\DD)$ and $\dual(\DD)$.  

For the crawling step, suppose that $\EE, \FF$ are adjacent in
$\CG(V)$ and that $\Phi$ fixes $\EE$, $\link(\EE)$ and $\dual(\EE)$.
Let $\GG$ be a pants decomposition obtained by adding the new disk of
$\FF$ to $\EE$ and then adding non-separating disks until we have $3g
- 3 + n$ disks.  Let $\{ P_k \}$ enumerate the solid pants of $\GG$.
Let $X_i = P_k \cup P_\ell$ be the four-spotted ball containing $G_i$
in its interior.

Let $\HH, \II = \{ H_i \}, \{ I_i \}$ be collections of disks so that
$H_i, I_i$ are contained in $X_i$ and $G_i, H_i, I_i$ are pairwise
dual in $X_i$.  Now, all of these disks $\GG \cup \HH \cup \II$ lie in
$\EE \cup \link(\EE) \cup \dual(\EE)$.  Thus $\Phi$ fixes all of them.
Thus $\Phi$ fixes $\FF$.  Consider $\Phi|\link(\FF)$.  By
\refthm{Korkmaz} the automorphism $f = \Phi|\link(\FF)$ is geometric.
Let $f$ also denote the given homeomorphism of $V' = V \setminus
\neigh(\FF)$.  Let $\GG' = \GG \setminus \FF$ and $\HH', \II'$ be the
disks of $\HH, \II$ contained in $V'$.  Thus $f$ fixes all disks of
$\GG', \HH', \II'$.  It follows that $f$ permutes the solid pants $\{
P_k \}$.

If $f$ nontrivially permutes $\{ P_k \}$ then, since each $G_i$ is
fixed, we find that adjacent solid pants are interchanged.  This
implies that $V' = P_1 \cup P_2$, a contradiction. 

So $f$ fixes every $P_k$.  Since all disks in $\GG'$ are fixed, $f$ is
either orientation reversing, isotopic to the identity, or isotopic to
a half-twist on each of the $P_k$.  Let $G_i \in \GG'$ be any disk
meeting $P_k$.  Then $f|P_k$ cannot be orientation reversing because
the triple $G_i, H_i, I_i$ determines an orientation on $X_i$ and
hence on $P_k$.  If $f|P_k$ is a half-twist then $P_k$ meets two spots
of $V'$.  Thus $G_i$ meets two solid pants $P_k, P_\ell$ so that $X_i
= P_k \cup P_\ell$.  Now, as $e(V') \geq 3$, the solid pants $P_\ell$
meets at most one spot of $V'$.  Thus $f|P_\ell$ is isotopic to the
identity.  So if $f|P_k$ is a half-twist then $f(H_i) \neq H_i$, a
contradiction.  Deduce that $f$, when restricted to any solid pants,
is isotopic to the identity.  Now, since $f$ fixes all of the $H_i$,
$f$ is isotopic to the identity on $V'$, as desired.

Deduce that $\Phi|\link(\FF)$ is the identity.  As $\Phi$ fixes duals
to $\FF$ by \reflem{CrawlDual} the automorphism $\Phi$ fixes all of
$\dual(\FF)$.  This completes the crawling step and so completes the
proof of:

\begin{theorem}
\label{Thm:AutOfDiskComplex}
If a handlebody $V = V_{g,n}$ satisfies $e(V) \geq 3$ then
the natural map $\calH(V) \to \Aut(\calD(V))$ is a surjection. \qed
\end{theorem}

As a corollary:

\begin{theorem}
\label{Thm:AutOfDiskComplexII}
If a handlebody $V = V_{g,n}$ satisfies $e(V) \geq 3$ then
the natural map $\calH(V) \to \Aut(\calD(V))$ is an isomorphism.
\end{theorem}

Note that Theorems~\ref{Thm:AutOfDiskComplex}
and~\ref{Thm:AutOfDiskComplexII} are sharp: when $e(V) \leq
2$ the conclusions are false.  See \refsec{Small}.  

\begin{proof}[Proof of \refthm{AutOfDiskComplexII}]
\refthm{AutOfDiskComplex} shows that the natural map is surjective.
Suppose that the mapping class $f$ lies in the kernel.  As in the
discussion of crawling through $\CG(V)$ given above, let $\GG = \{ G_i
\}$ be a pants decomposition of $V$ so that all of the $G_i$ are
non-separating.  Let $\{ P_k \}$ enumerate the solid pants of this
decomposition.  Let $X_i = P_j \cup P_k$ be the four-spotted ball
containing $G_i$ in its interior.  Let $\HH, \II = \{ H_i \}, \{ I_i
\}$ be collections of disks so that $H_i, I_i$ are contained in $X_i$
and $G_i, H_i, I_i$ are pairwise dual in $X_i$.  All of these disks
are fixed by $f$.  It follows that $f$ is isotopic to the identity.
\end{proof}

\section{An application}
\label{Sec:AutOfHandlebodyGroup}

\begin{theorem}
\label{Thm:AutOfHandlebodyGroup}
If $e(V) \geq 3$ then the outer automorphism group of the
handlebody group is trivial.
\end{theorem}

This may be restated as: $\Aut(\calH) \isom \calH$.  When $g = 0$ then
\refthm{AutOfHandlebodyGroup} follows from \reflem{SpottedBalls} and
the first author's thesis~\cite[Theorem 3]{Korkmaz99}.  For the rest
of this section we restrict to the case $g \geq 1$.

The idea of the proof is to turn an element $\phi \in \Aut(\calH)$
into an automorphism of the disk complex $\calD(V)$.  We do this,
following~\cite{Ivanov88}, by giving an algebraic characterization of
first Dehn twists about non-separating disks and then Dehn twists
generally.  We then apply Theorem~\ref{Thm:AutOfDiskComplex} to $\phi$
to find the correspoding geometric automorphism.  An algebraic trick
then gives the desired result.

A finite index subgroup $\Gamma < \calH$ is {\em pure} if every
reducible class in $\Gamma$ fixes every component of every reducing
set.  For example, the kernel of $\calH \to \Aut(H_1(\bdy_+ V,
\ZZ/3\ZZ))$ is pure.

\begin{lemma}
\label{Lem:NonSeparatingTwists}
Suppose $\Gamma < \calH$ is pure and finite index.  Then $\{f_i\}
\subset \calH$ is a collection of Dehn twists along a pants
decomposition of non-separating disks in $V$ if and only if
\begin{itemize}
\item
the subgroup $A = \langle f_i \rangle$ is free Abelian of rank $\xi(V)$,
\item
$f_i$ and $f_j$ are conjugate in $\calH$, for all $i, j$,
\item
$f_i$ is {\em primitive} in $C_\calH(A)$: $f_i$ is not a proper power
of any $h \in C_\calH(A)$, and
\item
the center of the centralizer of the class $f_i^n$ in $\Gamma$ is
infinite cyclic (for all $i$ and for all $n$ so that $f_i^n \in
\Gamma$): 
$$
C(C_\Gamma(f_i^n)) \isom \ZZ.
$$
\end{itemize}
\end{lemma}

\begin{proof}
The forwards direction is identical to the forwards direction
of~\cite[Theorem 2.1]{Ivanov88}.  The backwards direction is similar
in spirit to the backwards direction of~\cite[Theorem 2.1]{Ivanov88}
but some details differ.  Accordingly we sketch the backwards
direction.

The mapping class $f_i$ can not be periodic or pseudo-Anosov as that
would contradict the first property.  Let $\Theta \subset S = \bdy_+
V$ be the {\em canonical reduction system} for the Abelian group
$A$~\cite{BirmanEtAl83}.  Let $\{ X_j \}$ be the components of $S
\setminus \neigh(\Theta)$ and let $\{ Y_k \}$ be the collection of
annuli $\Neigh(\Theta)$.  By \cite[Lemma 3.1(2)]{BirmanEtAl83} the
number of annuli in $\{ Y_k \}$ plus the number of non-pants in $\{
X_j \}$ equals $\xi(V)$.  It follows that every non-pants $X_j$ has
complexity one (so is homeomorphic to $S_{0,4}$ or $S_{1,1}$).

Fix a power $n$ (independent of $i$) to ensure that $f_i^n \in
\Gamma$.  For each $X_j$ of complexity one there is some $f_i^n$ so
that $f_i^n|X_j$ is pseudo-Anosov.  Suppose that $f = f_1^n$, $X =
X_1$ has complexity one, and $f|X$ is pseudo-Anosov.  Let
$\lambda^\pm$ be the stable and unstable laminations of $f|X$.  For
every $i$, the mapping $f_i^n|X$ is either the identity or
pseudo-Anosov.  Note that in the latter case the stable and unstable
laminations of $f_i^n|X$ agree with $\lambda^\pm$: otherwise a
ping-pong argument gives a rank two free group in $A$, a
contradiction.  Thus, perhaps taking a larger power $n$, we may assume
that for each $i$ either $f_i^n|X$ is the identity or identical to
$f|X$.
For each $i$ where $f_i^n|X = f|X$ we temporarily replace $f_i$ by
$f_i^n f^{-1}$.

Continuing in this manner we find a free Abelian group $B < A \cap
\Gamma$ of rank at least
$|\Theta|$ where all elements are supported inside of the union of
annuli $\{ Y_k \}$.  Since $B$ is pure, it follows that all elements
of $B$ are compositions of powers of Dehn twists along disjoint
curves.  A theorem of McCullough~\cite{McCullough02} implies that
every curve in $\Theta$ either bounds a disk or cobounds an annulus
with some other curve of $\Theta$.  However, each annulus reduces the
possible rank of $B$ by one; it follows that every curve in $\Theta$
bounds a disk.

Let $\gamma$ be any essential non-peripherial component of $\bdy X$.
It follows that $f$ commutes with $T_\gamma$, that $T_\gamma$ lies in
$\calH$ by the above paragraph, and that $T_\gamma$ to some power lies
in $C(C_\Gamma(f))$.  But this contradicts the fourth property.  It
follows that every component $X_j$ is a pants and that $|\Theta| =
\xi(V)$.  Thus every $f_i$ is a compositions of powers of disjoint
twists.  Again, by the fourth property each $f_i$ is some power of a
single twist.  By the third property (following~\cite{Ivanov88}) $f_i$
is in fact a twist.  Finally, by the second property, each twist is
supported on a disk of the same topological type.  As every pants
decomposition of $V$ must contain a non-separating disk all of the
twists $f_i$ are supported by non-separating disks.
\end{proof}

We now give the general characterization:

\begin{lemma}
\label{Lem:GeneralTwists}
Suppose $\Gamma < \calH$ is pure and finite index.  Then $\{f_i\}
\subset \calH$ is a collection of Dehn twists along a pants
decomposition of $V$ if and only if
\begin{itemize}
\item
the subgroup $A = \langle f_i \rangle$ is free Abelian of rank
$\xi(V)$,
\item
$f_i$ is primitive in $C_\calH(A)$,
\item
for all $i$ and for all $n$ so that $f_i^n \in \Gamma$ either
$C(C_\Gamma(f_i^n)) \isom \ZZ$ or there is a $j$ so that
$C(C_\Gamma(f_i^n)) \isom \ZZ^2$ with the latter given by $\langle
f_i, f_j \rangle$ and $f_j$ is a twist on a non-separating disk.
\end{itemize}
\end{lemma}

\begin{proof}
Suppose that $\{ D_i \}$ is a pants decomposition and $f_i$ is the
positive twist on $D_i$.  Then $A = \langle f_i \rangle$ is free
Abelian of the correct rank.  The second property follows as $A =
C_\calH(A)$.  The third property follows from Ivanov's
discussion~\cite{Ivanov88} except if $D_i$ is a handle disk.  In this
case the meridian of the handle, say $D_j$, gives a twist $f_j$ which
lies in the center of the centralizer.

The backwards direction is similar to that of the proof of
\reflem{NonSeparatingTwists}.  The only change occurs when $f|X$ is
pseudo-Anosov: when the center of the centralizer has rank two then
the additional element is a twist about a separating disk and this
contradicts the third property.
\end{proof}

The following lemmas follow from the idential statements for the
mapping class group of $S$~\cite{IvanovMcCarthy99}:

\begin{lemma}
\label{Lem:FarCommutivity}
Suppose $D$ and $E$ are essential disks.  The twists $T_D, T_E$
commute if and only if $D$ and $E$ can be made disjoint via proper
isotopy. \qed
\end{lemma}

\begin{lemma}
\label{Lem:Conjugation}
For any twist $T_D$ and for any homeomorphism $h$ we have $h T_D
h^{-1} = T_{h(D)}$. \qed
\end{lemma}

\begin{lemma}
\label{Lem:Powers}
For any pair of disks $D$ and $E$ and any pair of integers $n$ and
$m$, if $T_D^n = T_E^m$ then $D = E$ and $n = m$. \qed
\end{lemma}

The proof of \refthm{AutOfHandlebodyGroup} now follows, essentially
line-by-line, the proof of either \cite[Theorem~2]{Ivanov97} or
\cite[Theorem~3]{Korkmaz99}.  \qed

\vspace{\baselineskip}

To extend our algebraic characterization of twists in $\calH(V)$
(Lemmas \ref{Lem:NonSeparatingTwists} and \ref{Lem:GeneralTwists}) to
a characterization of powers of twists inside of finite index pure
subgroups $\Gamma < \calH(V)$ appears to be a delicate matter.
Solving this problem would, following Ivanov~\cite{Ivanov97}, solve:

\begin{problem}
\label{Prob:Comm}
Show that the abstract commensurator of $\calH(V)$ is $\calH(V)$
itself.  Show that $\calH(V)$ is not arithmetic. 
\end{problem}

\bibliographystyle{plain} 
\bibliography{bibfile}
\end{document}